\documentclass[12pt, reqno]{amsart}
\usepackage{amssymb, amsmath, amsfonts}
\usepackage[T1]{fontenc}
\usepackage[utf8]{inputenc}
\usepackage[all]{xy}
\usepackage{xcolor}
\usepackage{amsthm, a4wide}
\usepackage{mathtools}
\usepackage{orcidlink}

\xyoption{all}

\usepackage{multicol}
\usepackage{hyperref}
\usepackage{graphicx}

\newtheorem{definition}{Definition}[section]
\newtheorem{theorem}[definition]{Theorem}
\newtheorem{proposition}[definition]{Proposition}
\newtheorem{lemma}[definition]{Lemma}
\newtheorem{corollary}[definition]{Corollary}
\newtheorem{remark}[definition]{Remark}
\newtheorem{example}[definition]{Example}

\newcommand{\lra}{\longrightarrow}

 \newcommand{\ele}{\mathcal L}  \newcommand{\ka}{\mathcal K}

\def\Lb{\operatorname{\textbf{\textsf{Lb}}}}
\def\nLie{\operatorname{{}_{n}\textbf{\textsf{Lie}}}}
\def\nLb{\operatorname{{}_{n}\textbf{\textsf{Lb}}}}
\def\qLb{\operatorname{{}_{q}\textbf{\textsf{Lb}}}}
\def\pLb{\operatorname{{}_{p}\textbf{\textsf{Lb}}}}
\def\XLb{\operatorname{\textbf{\textsf{X}}\textbf{\textsf{Lb}}}}
\def\XnLb{\operatorname{\textbf{\textsf{X}}{}_{n}\textbf{\textsf{Lb}}}}
\def\XqLb{\operatorname{\textbf{\textsf{X}}{}_{q}\textbf{\textsf{Lb}}}}
\def\XpLb{\operatorname{\textbf{\textsf{X}}{}_{p}\textbf{\textsf{Lb}}}}
\def\X2Lb{\operatorname{\textbf{\textsf{X}}{}_{2}\textbf{\textsf{Lb}}}}
\def\Set{\operatorname{\textbf{\textsf{Set}}}}

\def\U{\operatorname{\frak{U}}}
\def\XU{\operatorname{\frak{XU}}}

\def\ele{\operatorname{\mathcal{L}}}
\def\I{\operatorname{\mathcal{I}}}
\def\Up{\operatorname{\mathcal{J}}}

\def\id{\operatorname{id}}

\DeclareMathOperator{\Ker}{\mathsf {Ker}}
\DeclareMathOperator{\Coker}{\mathsf {Coker}}
\DeclareMathOperator{\Img}{\mathsf {Im}}

\DeclareMathOperator{\Z}{\mathsf {Z}}

\begin{document}
\title[Notes on Leibniz $n$-algebras]{Notes on Leibniz $n$-algebras}

\author[J. M. Casas]{Jos\'e Manuel Casas \orcidlink{0000-0002-6556-6131}}
\address{[J. M. Casas] Department of Applied Mathematics I, and CITMAga, E. E. Forestal,  University of Vigo, 36005 Pontevedra, Spain}
\email{jmcasas@uvigo.es}

\author[E. Khmaladze]{Emzar Khmaladze \orcidlink{0000-0001-9492-982X}}
\address{[E. Khmaladze] The University of Georgia, Kostava St. 77a, 0171 Tbilisi, Georgia \& A. Razmadze Mathematical Institute of Tbilisi State University,
	Tamarashvili St. 6, 0177 Tbilisi, Georgia   }
\email{e.khmaladze@ug.edu.ge}

\author[M. Ladra]{Manuel Ladra \orcidlink{0000-0002-0543-4508}}
\address{[M. Ladra] Department of Mathematics, CITMAga, Universidade de Santiago de Compostela, 15782 Santiago de Compostela, Spain}
\email{manuel.ladra@usc.es}

\begin{abstract}
We analyze behaviors of generalized forgetful and Daletskii-Takhtajan's functors on perfect objects and crossed modules of Leibniz $n$-algebras. Then we give applications to homology and universal central extensions of Leibniz $n$-algebras.
\end{abstract}



\maketitle

\def\Lb{\operatorname{\textbf{\textsf{Lb}}}}
\def\nLb{\operatorname{{}_{n}\textbf{\textsf{Lb}}}}
\def\qLb{\operatorname{{}_{q}\textbf{\textsf{Lb}}}}
\def\pLb{\operatorname{{}_{p}\textbf{\textsf{Lb}}}}
\def\Set{\operatorname{\textbf{\textsf{Set}}}}

\def\U{\operatorname{\frak{U}}}
\def\D{\operatorname{\frak{D}}}

\def\ad{\operatorname{\text{ad}}}

\numberwithin{equation}{section}

\section{Introduction}

Leibniz $n$-algebras were introduced in \cite{CLP} as the non-skew symmetric version of Nambu algebras \cite{DT, Na} or Lie $n$-algebras \cite{Fi}.
For $n=2$ the Leibniz $2$-algebras are ordinary Leibniz algebras. Thus, Leibniz $n$-algebras are simultaneous generalization of Lie, Leibniz and Lie $n$-algebras.

In the development of structure theory of Leibniz $n$-algebras one often finds analogies with results established in Lie $n$-algebras and Leibniz algebras. One such direction of the study is the use of \emph{forgetful functor} $\U_n$ which assigns to a Leibniz algebra $\ele$ the Leibniz $n$-algebra with the same underlying vector space $\ele$ endowed with the $n$-ary operation given by $[l_1, ..., l_n] = [l_1,[l_2, ...,[l_{n-1},l_n]\dots]]$.


On the other hand, in the investigation of homological aspects of Leibniz $n$-algebras treated in \cite{Ca, Ca 1, Ca 2}, one exploits the
remarkable properties of the so-called \emph{Daletskii-Takhtajan's functor}  $\D_n$ which assigns to a Leibniz $n$-algebra $\ele$ the Leibniz algebra with the underlying vector space $\ele^{\otimes(n-1)}$ and with the bracket operation given by 
\[
[l_1\otimes\cdots\otimes l_{n-1}, l'_1\otimes\cdots\otimes
l'_{n-1}]=\underset{\scriptsize{1\leq i\leq n}}\sum
l_1\otimes\cdots\otimes
[l_i,l'_1,\dots,l'_{n-1}]\otimes\cdots\otimes l_{n-1}.
\]


In this paper we choose to establish further properties of the generalized forgetful and Daletskii-Takhtajan's functors between categories of Leibniz $n$-algebras and Leibniz $p$-algebras, where $p=k(n-1)+1$ for some natural number $k$. In particular, we prove that the generalized forgetful functor preserves perfect objects (Theorem \ref{Theorem perfect}), actions (Lemma \ref{Lemma action}) and crossed modules (Proposition \ref{Prop U preserves CM}).
In Section \ref{Section D functor}, we show that the Daletskii-Takhtajan's functor does not preserve perfect objects in general, by presenting counter-examples (Example \ref{Example countre}).
In Section \ref{Section homology}, we construct a homomorphism between homologies with coefficients in a trivial co-representation of a Leibniz algebra $\ele$ and the Leibniz $n$-algebra $\U_{n}(\ele)$. And finally, as an application, in Section \ref{Section UCE}, we analyze the relationships between the universal central extensions of a perfect Leibniz $n$-algebra in the categories of Leibniz $n$-algebras and Leibniz $p$-algebras.

\subsection*{Conventions and notation.} In the whole paper $\mathbb{K}$ denotes the ground field and all tensor products are taken over $\mathbb{K}$. Linear maps are $\mathbb{K}$-linear maps and vector spaces are $\mathbb{K}$-vector spaces as well.  $n$ is a natural number, $n\geq 2$ and $p=k(n-1)+1$, $q=\kappa(p-1)+1=k\kappa(n-1)+1$ for some natural numbers $k$, $\kappa$.

\section{Leibniz $n$-algebras}

The following definition first appeared in \cite{CLP}.

\begin{definition}  A Leibniz $n$-algebra  is a vector space ${\ele}$
equipped with an $n$-linear operation ($n$-ary bracket) $[-, \dots,-] : {\ele}^{\otimes n} \to {\ele}$  satisfying the following fundamental identity for all $x_1, \dots, x_n, y_2, \dots, y_n \in {\ele}$:
\begin{equation} \label{FI}
\begin{array}{l}
[[x_1, x_2, \dots, x_n],y_2,\dots,y_{n}] =
  \displaystyle \sum_{i=1}^n [x_1,\dots, x_{i-1}, [ x_i,y_2,\dots,
y_{n}],x_{i+1},\dots,x_n]
\end{array}
\end{equation}
\end{definition}

In addition, if the $n$-ary bracket is skew-symmetric, that is
\[
[x_{\sigma(1)}, \dots, x_{\sigma(n)}] = {\text{sgn}(\sigma)} [x_1, \dots, x_n],
\]
 for all $\sigma \in S_n$, then ${\ele}$ is said to be \emph{a  Lie $n$-algebra or Filippov algebra } \cite{Fi, Na}. Here $S_n$ stands for the permutation group on $n$ elements and $\text{sgn}(\sigma)\in \{-1,\ 1\}$ denotes the
signature of $\sigma$.

\emph{A homomorphism} of Leibniz $n$-algebras is a linear map preserving the $n$-ary bracket. The respective category of Leibniz $n$-algebras will be denoted by $\nLb$.

\

For $n=2$ the identity (\ref{FI}) is equivalent to the Leibniz identity  $[x,[y,z]] = [[x,y],z] - [[x,z],y]$, so a Leibniz $2$-algebra is simply a Leibniz algebra \cite{Lo}. We write $\Lb$ for $_2{\Lb}$.

\

\begin{example} \
\begin{enumerate}
\item[i)] Any vector space with the trivial $n$-ary bracket is a Leibniz $n$-algebra, called an abelian Leibniz $n$-algebra.
\item[ ii)] Lie triple systems \cite{Li} and Leibniz triple systems \cite{BSO} are (non-Lie) Leibniz $3$-algebras.

\item[ iii)] An associative trialgebra \cite{LR} with the $3$ binary operations $\dashv$, $\vdash$  and $ \perp$, endowed with the ternary bracket
\[
[x,y,z] = x \dashv (y \perp z -z \perp y) - (y \perp z -z \perp
y) \vdash x,
\]
is (in general ) a non-Lie Leibniz $3$-algebra.
\item[ iv)]  Let $\text{char}(\mathbb{K})\mid (n-1)$. Then any one-dimensional vector space with basis $\{ e\}$ is a Leibniz $n$-algebras with respect the $n$-ary bracket given by $[e,e, \dots, e]= \alpha  e$, for any $\alpha \in \mathbb{K}$ (c.f. \cite[Example 2.1 (iv)]{CCGLO}).
\item[v)] Any two-dimensional vector space with
basis $\{e_1, e_2\}$ is a Leibniz $3$-algebras with respect to the ternary bracket given by  $[e_1,e_2,e_2]=e_1$ and $0$ otherwise (c.f. \cite[Theorem 2.14]{CCGLO}).

\end{enumerate}
\end{example}

A subalgebra ${\ele}'$ of a Leibniz $n$-algebra ${\ele}$ is called {\it $n$-sided ideal} of ${\ele}$ if $[l_1, \dots, l_n]
\in {\ele}'$  as soon as $l_i \in {\ele}'$ for some $i$, $1\leq i\leq n$.

For any two $n$-sided ideals ${\ele}'$ and ${\ele}''$  of a Leibniz $n$-algebra ${\ele}$, we denote by $[{\ele}', {\ele}'', {\ele}^{n-2}]$ {\it the commutator} of
${\ele}'$ and ${\ele}''$, that is, the $n$-sided ideal of ${\ele}$ spanned by the brackets $[l_1,\dots,l_i, \dots, l_j,\dots,l_n]$, where necessarily $l_i \in {\ele}'$ and $l_j\in {\ele}''$ for some $i$ and $j$, $1\leq i, j\leq n$, $i\neq j$. Obviously $[{\ele}', {\ele}'', {\ele}^{n-2}] \subseteq {\ele}' \cap {\ele}''$. In particular, if ${\ele}''={\ele}$ (resp. ${\ele}'={\ele}''={\ele}$ ), then we use the notation $[{\ele}', {\ele}^{n-1}]$ (resp. $[{\ele}^n]$) instead of $[{\ele}', {\ele}, {\ele}^{n-2}]$ (resp. $[{\ele}, {\ele}, {\ele}^{n-2}]$).

\

Given a Leibniz $n$-algebra $\mathcal{L}$, it is clear that the quotient of $\mathcal{L}$ by the $n$-sided ideal generated by all elements of the form
\[
[l_1, \dots , l_n]- {\text{sgn}(\sigma)}[l_{\sigma(1)}, \dots ,l_{\sigma(n)}],
\]
 where $l_1, \dots , l_n \in \mathcal{L}$ and $\sigma \in S_n$, is a Lie $n$-algebra, denoted by $\mathcal{L}_{\small{\text{Lie}}}$. This defines a functor
$(-)_{\small{\text{Lie}}}$ from $\nLb$ to the category $\nLie$ of Lie $n$-algebras, which is left adjoint to the full embedding functor $\nLie \hookrightarrow \nLb$.

\begin{definition}\
\begin{enumerate}
\item[i)] The {\it center}  of a Leibniz $n$-algebra ${\ele}$ is the $n$-sided ideal
\[
\Z({\ele}) = \{l \in {\ele} \mid [l_1, \dots, l_{i-1},l,l_{i+1}, \dots, l_n] = 0, \ \text{where} \  1 \leq i \leq n,  l_j \in {\ele}, j = 1, \dots, \hat{i}, \dots, n \}.
\]

\item[ii)] A Leibniz $n$-algebra ${\ele}$ is said to be {\it perfect} if ${\ele} = [{\ele}^n]$.
\end{enumerate}
\end{definition}

\section{ $\U_p^q$ functors and perfect objects}\label{Section U perfect}

Any Leibniz algebra is also Leibniz $n$-algebra with respect to
the $n$-ary bracket $$[x_1,x_2,\dots,x_n]=[x_1,[x_2,\dots,[x_{n-1},x_n]
\cdots ]]$$ (see \cite[Proposition 3.2]{CLP}). This provides the so-called forgetful functor
\[
\U^n:\Lb\to \nLb, \quad {\ele} \mapsto {\ele}.
\]

In this section we deal with generalizations of the functor $\U^n$ to the categories $\pLb$ ($p=k(n-1)+1$) and  $\qLb$ ($q=\kappa(p-1)+1$).
We need to recall the following definition from \cite{CLP}.
\begin{definition}
Let $A$ be a vector space equipped with an $n$-linear operation $\omega : A^{\otimes n} \to A$. A map
$f:A \to A$ is said to be a derivation with respect to $\omega$  if
\[
f(\omega(a_1, \dots,a_n)) = \sum_{i=1}^n  \omega(a_1, \dots, f(a_i), \dots,a_n).
\]
\end{definition}

Note that, given a vector space ${\ele}$, an $n$-linear operation $[-,\stackrel{n}\dots,-] : {\ele}^{\otimes n} \to {\ele}$ satisfies the fundamental identity (\ref{FI})  if and only if, for all $x_2,\cdots, x_{n}\in {\ele}$, the map
\begin{equation}\label{ad}
\ad_{x_2,\cdots, x_{n}}:{\ele}\to {\ele}, \qquad x\mapsto [x,x_2,\cdots, x_{n}],
\end{equation}
is a derivation with respect to $[-,\stackrel{n}\dots,-]$.


\begin{proposition} \label{derivation}
Let $A$ be a vector space and $\omega_n : A^{\otimes n} \to A$ be an $n$-linear operation. Let $\omega_p :A^{\otimes p} \to A$
be given by
\begin{equation} \label{structure}
\omega_p(a_1,a_2,\dots,a_{p})
  =\omega_n\Big(a_1, \dots a_{n-1},\omega_n\big(a_n, \dots, a_{2n-2}, \omega_n(\dots,\omega_n(a_{p-n+1}, \dots, a_{p}) \dots)\big)\Big).
\end{equation}
Then, a derivation $f:A\to A$ with respect to $\omega_n$ is a derivation with respect  to $\omega_p$.
\end{proposition}
\begin{proof} This follows by the following straightforward calculations:
\begin{align*}
f(\omega_p &(a_1,\dots,a_{p})) \\
= & f\Big(\omega_n\Big(a_1, \dots a_{n-1},\omega_n\big(a_n, \dots, a_{2n-2}, \omega_n(\dots,\omega_n(a_{p-n+1}, \dots, a_{p}) \dots)\big)\Big)\Big)\\
 = & \omega_n\Big(f(a_1), \dots a_{n-1},\omega_n\big(a_n, \dots,a_{2n-2},\omega_n( \dots,\omega_n(a_{p-n+1}, \dots, a_{p}) \dots)\big)\Big) \\
  &+\cdots +\omega_n\Big(a_1, \dots f(a_{n-1}),\omega_n\big(a_n, \dots,a_{2n-2},\omega_n( \dots,\omega_n(a_{p-n+1}, \dots, a_{p}) \dots)\big)\Big) \\
   &+\omega_n\Big(a_1, \dots a_{n-1},f\big(\omega_n(a_n, \dots,a_{2n-2},\omega_n( \dots,\omega_n(a_{p-n+1}, \dots, a_{p}) \dots))\big)\Big) \\
 = &\dots = \sum_{i=1}^{p} \omega_p(a_1, \dots, f(a_i), \dots,a_{p}).
\end{align*}
\end{proof}


\begin{proposition} \label{leib}
Let ${\ele}$ be a Leibinz $n$-algebra with respect to the $n$-ary bracket \linebreak $[-,\stackrel{n}\dots,-] : {\ele}{^{\otimes n}} \to \ele$. Then ${\ele}$ is a Leibniz $p$-algebra  with respect to the $p$-ary bracket
\[
[l_1,l_2,\dots,l_{p}]
  =[l_1, \dots l_{n-1},[l_n, \dots, l_{2n-2}, [\dots,[l_{p-n+1}, \dots, l_{p}] \dots]]].
\]
\end{proposition}
\begin{proof} We know that, for any $x_2, \dots, x_{n} \in \ele$,  $\ad_{x_2, \dots, x_{n}} : {\ele} \to {\ele}$ from (\ref{ad})
is a derivation with respect to the $n$-ary bracket.
We must show that, for all $x_2, \dots, x_{p}\in \ele$,  $\ad_{x_2, \dots, x_{p}} : {\ele} \to {\ele}$ is also a derivation with respect to the described $p$-linear bracket. Immediately by definition of the $p$-linear bracket, we get
\[
\ad_{x_2, \dots, x_{p}}=\ad_{x_2, \dots x_{n-1},[x_{n}, \dots, x_{2n-2}, [\dots,[x_{p-n+1}, \dots, x_{p}] \dots]]}.
\]
Then the result follows by Proposition \ref{derivation}.
\end{proof}

\begin{remark} Let us observe that Proposition \ref{derivation} and Proposition \ref{leib} in the particular case $n=2$ give Proposition 2.2 and Proposition 3.2 in \cite{CLP}, respectively.
\end{remark}

As a consequence of Proposition \ref{leib}, there exists a functor
\[
\U_n^p:\nLb \to \pLb, \quad {\ele}\mapsto {\ele}.
\]
It is clear that, in particular, $\U_2^n$ is the same as  $\U^n:\Lb\to \nLb$ considered above. Moreover, we have the following commutative diagram of categories and functors
\begin{equation}\label{cd1}
\xymatrix{
	\Lb \ar[r]^{\U^n}  \ar[d]_{\U^p} & \nLb \ar[d]^{\U_n^q}  \\
	\pLb \ar[r]_{\U_p^q} & \qLb .
}
\end{equation}

\begin{theorem} \label{Theorem perfect}
$\U_n^p:\nLb\to \pLb$ is an exact functor preserving and reflecting perfect objects.
\end{theorem}
\begin{proof}
Exactness of $\U_n^p$ is obvious. Hence, we have to prove that a Leibniz $n$-algebra ${\ele}$ is perfect  if and only if $\U_n^p({\ele})$ is a perfect Leibniz $p$-algebra.

First suppose that ${\ele}$ is a perfect Leibniz $n$-algebra, i.e. ${\ele} =[{\ele}^n]$. Then any element $x\in {\ele}$ has the form $x = \sum \lambda_1 [x_{11}, \dots,x_{1n}]$. By the same argument,  the element $x_{1n}\in \ele$ has the form
 $x_{1n} = \sum \lambda_2 [x_{21}, \dots,x_{2n}]$ and so on, after $k$ steps (recall that $p=k(n-1)+1$) we get
 \[
 x= \sum \lambda_1 \lambda_2 \dots \lambda_{k} \omega_p (x_{11}, \dots, x_{1,n-1}, x_{21}, \dots , x_{2,n-1}, \dots, x_{k1},\dots, x_{kn}  )\in [{\ele}^p],
 \]
  where $\omega_p:{\ele}^{\otimes p}\to {\ele}$ is defined in the same way as in (\ref{structure}), with $\omega_n(x_1, \dots,x_n) = [x_1, \dots, x_n]$.

Conversely, by definition of $p$-ary bracket on  $\U_n^p({\ele})$, it is clear that $[{\ele}^p]\subseteq [{\ele}^n] \subseteq \ele$. So, if
$\ele = [ {\ele}^p ]$, then $\ele=[ {\ele}^n ]$.
\end{proof}

As a trivial consequence of Theorem \ref{Theorem perfect}, we deduce that each functor in the diagram  (\ref{cd1}) preserves perfect objects. In particular, if  ${\ele}$ is a perfect Leibniz algebra, then ${\ele}$ is a perfect Leibniz $3$-algebra, a perfect Leibniz $4$-algebra and so on.

\section{ $\U_p^q$ functors and crossed modules}

In this section we show that the functor $\U_p^q$ can be extended to the respective categories of crossed modules. Before that, let's recall some necessary notions about actions and crossed modules of Leibniz $n$-algebras from \cite{CKL1}.

\begin{definition}\label{Def action} We say that a Leibniz $n$-algebra $\mathcal{P}$ acts on
a Leibniz $n$-algebra $\mathcal{L}$ if $2^n-2$ $n$-linear maps
\[
[-,\stackrel{n}{\cdots},-]:\mathcal{L}^{\otimes i_1}\otimes \mathcal{P}^{\otimes
j_1}\otimes\cdots\otimes \mathcal{L}^{\otimes i_m}\otimes
\mathcal{P}^{\otimes j_m}\to \mathcal{L}
\]
are given, where $1\leq m\leq n-1$, $\underset{\scriptsize{1\leq
k\leq m}}\sum (i_k+j_k)=n$, $0\leq i_k \leq n-1$ and at least one
$i_k\neq 0$, $0\leq j_k \leq n-1$ and at least one $j_k\neq 0$,
such that $2^{2n-1}-2$ equalities hold which are obtained from the fundamental identity
(\ref{FI}) by taking exactly $i$ of the variables
$x_1,\dots,x_n,y_1,\dots, y_{n-1}$ in $\mathcal{L}$ and all the
others in $\mathcal{P}$, 
and by changing $i=1,\dots,2n-2$.

\end{definition}

For example, if $\mathcal{L}$ is an $n$-sided ideal of the Leibniz $n$-algebra $\mathcal{P}$, then the
$n$-ary bracket in $\mathcal{P}$ yields an action of $\mathcal{P}$ on $\mathcal{L}$.

\begin{lemma}\label{Lemma action}
Let $\mathcal{P}$ and $\mathcal{L}$ be Leibniz $n$-algebras and $\mathcal{P}$ acts on $\mathcal{L}$. Then the  Leibniz $p$-algebra $\U_n^p(\mathcal{P})$ acts on the Leibniz $p$-algebra $\U_n^p(\mathcal{L})$.
\end{lemma}
\begin{proof}
Since $\U_n^p(\mathcal{P})=\mathcal{P}$ and $\U_n^p(\mathcal{L})= \mathcal{L} $ are Leibniz $p$-algebras with the $p$-ary brackets as described in Proposition \ref{leib}, we need to construct $2^p-2$ $p$-linear maps
\[
[-,\stackrel{p}{\cdots},-]:\mathcal{L}^{\otimes i_1}\otimes \mathcal{P}^{\otimes
j_1}\otimes\cdots\otimes \mathcal{L}^{\otimes i_m}\otimes
\mathcal{P}^{\otimes j_m}\to \mathcal{L},
\]
with $\underset{\scriptsize{1\leq
k\leq m}}\sum (i_k+j_k)=p$ and satisfying the same conditions as in Definition \ref{Def action}. Let us fix such $i_1, \cdots, i_m$, $j_1, \cdots, j_m$
and define the corresponding $p$-linear map $[-,\stackrel{p}{\cdots},-]$ by
\begin{align}\label{equation action}
[x_1, & \dots, x_{i_1}, x_{i_1 +1}, \dots, x_{i_1+j_1}, \dots,   x_{p}] \notag \\
 & =[x_1, \dots x_{n-1},[x_n, \dots, x_{2n-2}, [\dots,[x_{p-n+1}, \dots, x_{p}] \dots]]],
\end{align}
where any $x$ element belongs to either $\mathcal{L}$ or $\mathcal{P}$, but at the same time, at least one of these elements belongs to $\mathcal{L}$ and at least one belongs to $\mathcal{P}$. In the right hand side of equation (\ref{equation action}) the iterated $n$-ary brackets are given by the action of $\mathcal{P}$ on $\mathcal{L}$ and therefore satisfy the fundamental equation (\ref{FI}). Then, by the same argument as in Proposition \ref{leib}, it is easy to see that $p$-linear brackets defined by (\ref{equation action}) also satisfy the corresponding fundamental identity, i.e. they indeed define an action of $\U_n^p(\mathcal{P})$ on $\U_n^p(\mathcal{L})$.
\end{proof}

\begin{definition}\cite{CKL1}\label{Def crossed mod}
A crossed module of Leibniz $n$-algebras is a homomorphism
$\mu:\mathcal{L}\to \mathcal{P}$ together with an action of
$\mathcal{P}$ on $\mathcal{L}$ satisfying the following
conditions:
\begin{enumerate}
\item[(CM1)] $\mu$ is compatible with the action of
$\mathcal{P}$ on $\mathcal{L}$, that is,
\begin{align*}
&\mu[l_1^1,\dots,l_{i_1}^1,x_1^1,\dots,x_{j_1}^1,\dots,l_1^m,\dots,l_{i_m}^m,x_1^m,\dots,x_{j_m}^m]\\&=
[\mu l_1^1,\dots,\mu l_{i_1}^1,x_1^1,\dots,x_{j_1}^1,\dots,\mu
l_1^m,\dots,\mu l_{i_m}^m,x_1^m,\dots,x_{j_m}^m]\;,
\end{align*}
for $0\leq i_k,j_k \leq n-1$, $0\leq k\leq m$, $1\leq m\leq n-1$
such that $\underset{\scriptsize{1\leq k\leq m}}\sum (i_k+j_k)=n$,
at least one $i_k\neq 0$ and at least one $j_k\neq 0$. Here and below, all $l$'s belong to $\mathcal{L}$ and all $x$'s belong to $\mathcal{P}$.
\item[(CM2)] For every $1\leq i \leq n-1$, the $n$-bracket
in $\mathcal{L}$
$$
(*)\ \ \ \ \  \ \ \ \ \ \ \ \  \ \ \  \ \ \ \ \  \ \ \ \ \
   [l_1,l_2,\dots,l_n ] \ \ \ \ \ \ \ \ \ \ \ \ \ \ \ \ \
\ \ \ \  \ \ \ \ \
$$
 is equal to any expression obtained from
$(*)$ by replacing exactly $i$ of the variables $l$'s by $\mu
l$'s.
\item[(CM3)] If $\underset{\scriptsize{1\leq k\leq
m}}\sum i_k\geq 2$, then the expression
$$
(**) \ \ \ \ \  \ \ \
[l_1^1,\dots,l_{i_1}^1,x_1^1,\dots,x_{j_1}^1,\dots,l_1^m,\dots,l_{i_m}^m,x_1^m,\dots,x_{j_m}^m]\
\ \
$$
is equal to any expression obtained from $(**)$ by replacing
exactly one of $l$'s (and so $i$ of $l$'s, for every $1\leq i\leq
\underset{\scriptsize{1\leq k\leq m}}\sum i_k -1$) by $\mu l$.
\end{enumerate}
\end{definition}

{\em A morphism of crossed modules}
$(\mathcal{L}\overset{\mu}\longrightarrow\mathcal{P})\to
(\mathcal{L'}\overset{\mu'}\longrightarrow\mathcal{P'})$ is a pair
$(\alpha,\beta)$, where $\alpha:\mathcal{L}\to \mathcal{L'}$ and
$\beta:\mathcal{P}\to \mathcal{P'}$ are homomorphisms of Leibniz
$n$-algebras satisfying:
\begin{enumerate}
\item[i)] $\mu'\alpha=\beta\mu$, \item[ii)] $\alpha$ preserves the action
of $\mathcal{P}$ via $\beta$, i.e.,
\begin{align*}
&\alpha[l_1^1,\dots,l_{i_1}^1,x_1^1,\dots,x_{j_1}^1,\dots,l_1^m,\dots,l_{i_m}^m,x_1^m,\dots,x_{j_m}^m]\\&=
[\alpha l_1^1,\dots,\alpha l_{i_1}^1,\beta x_1^1,\dots,\beta
x_{j_1}^1,\dots,\alpha l_1^m,\dots,\alpha l_{i_m}^m,\beta x_1^m,\dots,\beta
p_{j_m}^m]\;.
\end{align*}
\end{enumerate}

Denote by ${\XnLb}$ the category of crossed modules of Leibniz $n$-algebras.
\
Note that for $n=2$ the conditions (CM2) and (CM3) are the same and we obtain the definition of crossed module of
Leibniz algebras \cite{LP}. The respective category will be denoted by $\XLb$ instead of ${\X2Lb}$.

\

There are full embedding functors
\[
\I_n^0, \I_n^1 \colon \nLb\to\XnLb
\]
defined by $\I_n^0(\ele)=(0\xrightarrow{ \ 0 \ } \ele)$, $\I_n^1(\ele)=(\ele\xrightarrow{\id_{\ele}} \ele)$
for a Leibniz $n$-algebra $\ele$.

Given a crossed module $\mu:\mathcal{L}\to \mathcal{P}$, then $\Img(\mu)$ is an $n$-sided ideal of $\mathcal{P}$. Let us define three functors
\[
\Up_n^0, \Up_n^1, \Up_n^2 : \XnLb \to \nLb
\]
by $\Up_n^0({\ele}\xrightarrow{\mu} \mathcal{P})=\mathcal{P}\slash \Img(\mu)$, $\Up_n^1({\ele}\xrightarrow{\mu} \mathcal{P})=\mathcal{P}$
 and $\Up_n^2({\ele}\xrightarrow{\mu} \mathcal{P})=\ele$.

\begin{remark}\label{remark_adj_Di_Lb}
	 It is straightforward to check that $\Up_n^i$ is left adjoint to $\I_n^i$ and $\I_n^i$ is left adjoint to
  $\Up_n^{i+1}$ for $i=0,1$.
\end{remark}

\begin{proposition}\label{Prop U preserves CM}
If $\mu:\mathcal{L}\to \mathcal{P}$ is a crossed module of Leibniz $n$-algebras, then ${\U}_n^p(\mu):{\U}_n^p(\mathcal{L})\to {\U}_n^p(\mathcal{P})$ is a crossed module of Leibniz $p$-algebras.
\end{proposition}
\begin{proof}
It can be readily checked that the homomorphism ${\U}_n^p(\mathcal{P})$, with the action of ${\U}_n^p(\mathcal{P})$ on ${\U}_n^p(\mathcal{L})$
described in Lemma \ref{Lemma action}, satisfies conditions (${}_n$CM1)-(${}_n$CM3) for crossed modules in $\pLb$. For instance, we only show that
${\U}_n^p(\mu)$ is compatible with the action of
${\U}_n^p(\mathcal{P})$ on ${\U}_n^p(\mathcal{L})$. In effect,
\begin{align*}
&{\U}_n^p(\mu)[l_1^1,\dots,l_{i_1}^1,x_1^1,\dots,x_{j_1}^1,\dots,l_1^m,\dots,l_{i_m}^m,x_1^m,\dots,x_{j_m}^m]\\
&={\U}_n^p(\mu) [y_1, \dots, y_{i_1}, y_{i_1 +1}, \dots, y_{i_1+j_1}, \dots,   y_{p}]  \\
 & = \mu[y_1, \dots y_{n-1},[y_n, \dots, y_{2n-2}, [\dots,[y_{p-n+1}, \dots, y_{p}] \dots]]]\\
&=[\mu(l_1^1),\dots,\mu(l_{i_1}^1),x_1^1,\dots,x_{j_1}^1,\dots,\mu(l_1^m),\dots,\mu(l_{i_m}^m),x_1^m,\dots,x_{j_m}^m].
\end{align*}
Here we used the following notations:

\noindent $y_1=l_1^1, \ \dots,\ y_{i_1}=l_{i_1}^1, \ y_{i_1+1}=x_1^1, \ \dots, \ y_{i_1+j_1} = x_{j_1}^1, \ \dots, \ y_{p-i_m-j_m+1}=l_{1}^m, \ \cdots, \ y_{p-j_m}=l_{i_m}^m, \  y_{p-j_m+1}=x_{1}^m, \ \dots, \  y_{p}=x_{j_m}^m$.
\end{proof}

The assignment in Proposition \ref{Prop U preserves CM} is functorial and we get a functor ${\XU}_n^p: \XnLb \to \XpLb $. Let us denote ${\XU}_2^n: \X2Lb \to \XnLb$ by ${\XU}^n$. By (\ref{cd1}) we get immediately the following commutative diagram of categories and functors
\begin{equation*}\label{cd2}
\xymatrix{
	\XLb \ar[r]^{\XU^n}  \ar[d]_{\XU^p} & \XnLb \ar[d]^{\XU_n^q}  \\
	\XpLb \ar[r]_{\XU_p^q} & \XqLb .
}
\end{equation*}

Furthermore, the functor $\XU_n^p$ is a canonical extension of $\U_n^p$ to the categories of crossed modules, in the sense of the following proposition.

\begin{proposition}
 The following diagrams of categories and functors
\begin{equation}\label{diagram I}
\xymatrix @=20mm {
 \nLb \ar[d]_{\I_n^i}\ar[r]^{\U_n^p} & \pLb \ar[d]^{\I_p^i}  \\
\XnLb \ar[r]_{\XU_n^p} & \XpLb
}
\qquad \qquad \qquad
\xymatrix @=20mm {
 \XnLb \ar[d]_{\Up_n^j}\ar[r]^{\XU_n^p} & \XpLb \ar[d]^{\Up_p^j}  \\
\nLb \ar[r]_{\U_n^p} & \pLb
}
\end{equation}
are commutative for $i=0, 1$ and $j=0,1,2$.
\end{proposition}
\begin{proof}
For $i\in \{ 0,1\}$ and $j\in \{1,2\}$, the equalities $\XU_n^p \circ \I_n^i= \I_p^i \circ \U_n^p$ and $\U_n^p \circ \Up_n^j= \Up_p^j \circ \XU_n^p$ are trivially verified.
To prove that $\XU_n^p \circ \Up_n^0= \Up_p^0 \circ \XU_n^p$, we need to show that
\[
\U_n^p(\mathcal{P})/\Img \U_n^p(\mu)= \U_n^p\big(\mathcal{P}/\Img(\mu) \big)
\]
for any crossed module $\mu:\mathcal{L} \to \mathcal{P}$. This last equality is an easy consequence of exactness of the functor $\U_n^p$ (see Theorem \ref{Theorem perfect}).

\end{proof}

\section{${\frak D}_q^p$ functors and perfect objects}\label{Section D functor}

 Given a Leibniz $n$-algebra ${\ele}$, we know that ${{\ele}}^{\otimes n-1}$ is a Leibniz algebra  with respect to the $n$-ary bracket
\[
[l_1\otimes\cdots\otimes l_{n-1}, l'_1\otimes\cdots\otimes
l'_{n-1}]=\underset{\scriptsize{1\leq i\leq n}}\sum
l_1\otimes\cdots\otimes
[l_i,l'_1,\dots,l'_{n-1}]\otimes\cdots\otimes l_{n-1}
\]
 (see \cite[Proposition 3.4]{CLP}). This defines the so-called Daletskii-Takhtajan's functor \cite{DT} (see also \cite{Ga})
 \[
 {\frak D}_{n}:\nLb \to \Lb, \quad {\ele} \mapsto {\ele}^{\otimes (n-1)}.
 \]

The generalized Daletskii-Takhtajan's functor  ${\frak D}_{q}^p$ between the categories  $\qLb$ and $\pLb$ is defined as follows (see \cite[Remark 3.5]{CLP}). Let $\ele$  be a Leibniz $q$-algebra. Taking in mind that $q=\kappa(p-1)+1$, then ${\ele}^{\otimes\kappa}$ becomes a Leibniz $p$-algebra with respect to the bracket given by
\[
[l_{11}\otimes \cdots \otimes l_{1\kappa}, \cdots, l_{p1}\otimes \cdots \otimes l_{p\kappa}] =
\underset{\scriptsize{1\leq i\leq \kappa}}\sum l_{11}\otimes \cdots \otimes  [l_{1i}, l_{21}, \cdots, l_{2\kappa},\cdots, l_{p1}, \cdots, l_{p\kappa}]\otimes \cdots \otimes l_{1\kappa}
\]
This correspondence defines a functor
\[
{\frak D}_q^p:\qLb \to \pLb, \quad {\ele}\mapsto {\ele}^{\otimes\kappa}.
\]
It is clear that, in particular, ${\frak D}_n^2 = {\frak D}_n:\nLb\to \Lb$ considered above, and we have the following commutative diagram of categories and functors
\begin{equation}\label{cd for D}
\xymatrix{
	\qLb \ar[r]^{{\frak D}_q^p}  \ar[d]_{{\frak D}_q^n} & \pLb \ar[d]^{{\frak D}_p}  \\
	\nLb \ar[r]_{{\frak D}_n} & \Lb .
}
\end{equation}

One of our main interests is to analyze the problem, which is in somewhat the reverse task of that in Section \ref{Section U perfect}, studying whether the functor  ${\frak D}_q^p$ preserves perfect objects. Let's concentrate to the case of the functor ${\frak D}_{n}={\frak D}_n^2$. The examples immediately below shows that even ${\frak D}_{n}$ does not preserve perfect objects in general.

\begin{example}\label{Example countre} \
\begin{enumerate}
\item[i)] Consider the four-dimensional  Lie $3$-algebra $\ele$ with basis $\{e_1, e_2, e_3, e_4 \}$ and ternary bracket given by
\[
[e_1,e_2,e_3]= e_4, \quad [e_1,e_2,e_4]= - e_3, \quad [e_1, e_3,e_4]= e_2, \quad [e_2,e_3,e_4]=-e_1,
\]
together with the corresponding skew-symmetry, and zero elsewhere \cite{Fi}. Obviously $\ele$ is a perfect Lie $3$-algebra, and so perfect Leibniz $3$-algebra. Nevertheless $\ele \otimes \ele$ is not a perfect Leibniz algebra since $\D_3(\ele)=[{\ele} \otimes {\ele}, {\ele} \otimes {\ele}]$
is generated by the following set of elements

\noindent
$\{ -e_1 \otimes e_3, \ -e_1 \otimes e_4, \  e_2 \otimes e_1, \  e_2 \otimes e_3, \ e_2 \otimes e_4, \ e_3 \otimes e_1, \ -e_3 \otimes e_2,\ - e_3 \otimes e_3, \  e_3 \otimes e_4, \ -e_4 \otimes e_1, \ e_4 \otimes e_2, \ e_4 \otimes e_3, \ e_4 \otimes e_4, \ -e_1 \otimes e_1 + e_2 \otimes e_2,\  e_1 \otimes e_2 + e_2 \otimes e_1, \  -e_1 \otimes e_3 - e_3 \otimes e_1, \  e_1 \otimes e_4 + e_4 \otimes e_1 \}, $

\noindent and the  elements $e_1 \otimes e_1, e_2 \otimes e_2 \in \ele \otimes \ele $
 can not be presented as a linear combination of commutators $[e_i \otimes e_j, e_k \otimes e_l],$ $i,j,k,l \in\{1,2,3,4\}$.

\item[ii)]  Let $\mathbb{K}$  contain the field of rational numbers and ${\ele}$ be the five-dimensional perfect Leibniz algebra with generators $\{e_1,e_2,e_3,e_4,e_5 \}$  and  bracket given by
\[
 \begin{array}{lll} [e_2,e_1]=-e_3, &
 [e_1,e_2]=e_3,& [e_1,e_3]=-2e_1 \\

 [e_3,e_1]=2e_1, & [e_3,e_2]=-2e_2, & [e_2,e_3]=2e_2 \\

 [e_5,e_1]=e_4, & [e_4,e_2]=e_5,& [e_4,e_3]=-e_4\\ &
 [e_5,e_3]=e_5 \end{array}
\]
 and zero elsewhere \cite{Om}. By Theorem \ref{Theorem perfect}, $\U_3(\ele)={\ele}$ is a perfect Leibniz $3$-algebra with ternary bracket operation given by
\[
\begin{array}{lll} [e_1,e_1,e_2]=-2e_1, &  [e_2,e_1,e_2]=2e_2,& [e_3,e_1,e_3]=-4e_1 \\

 [e_1,e_2,e_1]=2e_1, & [e_2,e_1,e_3]=2e_3, & [e_3,e_2,e_3]=-4e_2 \\

 [e_1,e_2,e_3]=2e_3, & [e_2,e_2,e_1]=-2e_2,& [e_3,e_3,e_1]=4e_1\\

 [e_1,e_3,e_2]=-2e_3,& [e_2,e_3,e_1]=-2e_3, & [e_3,e_3,e_2]= 4e_2 \\

 [e_4,e_1,e_2]=-e_4,& [e_5,e_1,e_2]=e_5 & \\

  [e_4,e_2,e_1]=e_4,& [e_5,e_1,e_3]=-2e_4 & \\

   [e_4,e_2,e_3]=2e_5,& [e_5,e_2,e_1]=-e_5 & \\

    [e_4,e_3,e_2]=-2e_5,& [e_5,e_3,e_1]=2e_4 & \\
 \end{array}
\]
 and zero elsewhere. It is easy to check that ${\frak D}_3({\ele}) = {\ele}^{\otimes 2}$ is not a perfect Leibniz algebra, since the basic elements $e_4 \otimes e_5$ and $e_5 \otimes e_4$ can not be written as a linear combination of commutators $[e_i \otimes e_j, e_k \otimes e_l]$, $i, j, k, l \in \{1, 2, 3, 4, 5 \}$.

 \item[iii)] Let  $\mathbb{K}$  contain the field of rational numbers and ${\ele}$ be the three-dimensional perfect Leibniz $3$-algebra with basis $\{e_1,e_2,e_3 \}$  and  ternary bracket  given by
\[
 \begin{array}{lll} [e_1,e_1, e_2]= \alpha e_1, & [e_2,e_1, e_2]= -\alpha e_2, & [e_3,e_1, e_2]= \beta e_3
  \end{array}
\]
 and zero elsewhere, where $\alpha$ and $\beta$ are non-zero elements in $\mathbb{K}$ \cite[Example 1.11]{Rustam}.
   It is easy to check that ${\frak D}_3({\ele}) = {\ele}^{\otimes 2}$ is not a perfect Leibniz algebra, since the basic elements $e_1 \otimes e_2$ and $e_2 \otimes e_1$ can not be written as a linear combination of commutators $[e_i \otimes e_j, e_k \otimes e_l]$, $i, j, k, l \in \{1, 2, 3 \}$.

\end{enumerate}
\end{example}

\section{A bit about the homology of Leibniz $n$-algebras}\label{Section homology}

In \cite{LP} the Leibniz homology $HL^*(\textrm{g},M)$ of a Leibniz algebra $\textrm{g}$
with coefficients in a co-representation $M$ of $\textrm{g}$ is computed to be the cohomology of the Leibniz chain complex
$CL_*(\textrm{g},M)$ given by
\[
CL_k(\textrm{g},M)
 = (M\otimes \textrm{g}^{\otimes k})\;, \ \ k\geq 0
 \]
 with the boundary operator $d_k:CL_k(\textrm{g},M)\to
 CL_{k-1}(\textrm{g},M)$ defined by
 \begin{align*}
 d_k(m \!\otimes\! x_1 \!\otimes\! \dots \!\otimes\! x_k) = ([m,x_1]\!\otimes\!x_2\!\otimes\! \dots \!\otimes\! x_k)
 + \underset{\scriptsize{2\leq i \leq k}}\sum (-1)^i ([x_i, m]\!\otimes\! x_1\!\otimes\! \dots \!\otimes\! \widehat{x_i} \!\otimes\! \dots  \!\otimes\! x_k)\\
 + \underset{\scriptsize{1\leq i < j \leq k}}\sum (-1)^{j+1}(m\!\otimes\! x_1\!\otimes\! \dots \!\otimes\! x_{i-1}\!\otimes\! [x_i,x_j]\!\otimes\!\dots \!\otimes\! \widehat{x_j} \!\otimes\! \dots \!\otimes\! {x_k} )\, .
\end{align*}

 It is shown in \cite{CLP} that the functor $\D_n$ sends a free Leibniz $n$-algebra to a free Leibniz algebra. This property of $\D_n$ is used in \cite{Ca 2, CLP} to construct explicit (co-)chain complex which computes Quillen (co)homology of a Leibniz $n$-algebra with coefficients.
In particular,  it is proved in \cite[Proposition 3.1]{Ca 2} that if $\mathcal {M}$ is a co-representation of a Leibniz $n$-algebra $\mathcal{L}$, then $\mathcal{M}\otimes \mathcal{L}$  can be considered as a co-representation of the Leibniz algebra $\D_n(\mathcal{L})=\mathcal{L}^{\otimes (n-1)}$.
Then a chain complex $_nCL_*(\mathcal{L}, \mathcal{M})$ is defined to be the Leibniz complex $CL_*(\D_n(\mathcal{L}), \mathcal{\mathcal{M}\otimes \mathcal{L}})$,
  and the homology ${}_nHL_*(\mathcal{L},\mathcal {M})$ of the Leibniz $n$-algebra $\mathcal{L}$ with coefficients in the co-representation $\mathcal{M}$  is defined by
\[
{}_nHL_*(\mathcal{L},\mathcal {M})=H_*\big({}_nCL_*(\mathcal{L},\mathcal{M})\big)
=HL_*\big(\D_{n}(\mathcal{L}),\mathcal{M}\otimes \mathcal{L}\big)\,.
\]

Note that for $n=2$ we have ${}_2CL_k(\mathcal {L},\mathcal
{M})\cong CL_{k+1}(\mathcal {L},\mathcal {M})$ for all $m\geq 0$ and hence
\begin{equation}\label{eq3}
{}_2HL_k(\mathcal {L},\mathcal {M})\cong HL_{k+1}(\mathcal{L},\mathcal {M})\,.
\end{equation}

In Section \ref{Section UCE} below, we will use the obvious fact that if $\mathcal {M}$ is a trivial co-representation of $L$, then ${}_nHL_0(\mathcal{L},\mathcal {M}) = \mathcal {M}\otimes \mathcal {L}/[\mathcal {L}^n]$. In particular, for $ \mathcal {M} = \mathbb{K}$ we get  ${}_nHL_0(\mathcal{L},\mathbb{K}) = \mathcal {L}/[\mathcal {L}^n]$.

\begin{proposition}\label{Prop homology complex }
If $\mathcal{M}$ is a trivial co-representation of a Leibniz algebra $\mathcal{L}$, then there is a morphism of chain complexes
\[
h_*: {}_{n}CL_*(\U_n(\mathcal{L}), \mathcal{M}) \lra {}_2CL_*(\mathcal{L}, \mathcal{M})\, ,
\]
 which is an isomorphism in dimension $0$.
Moreover, if  $\mathcal{L}$ is a perfect Leibniz algebra, then $h_*$ is an epimorphism.
\end{proposition}
\begin{proof}
First note that $\mathcal{M}$ is considered as a trivial co-representation of the Leibniz $n$-algebra $\U_n(\mathcal{L})=\mathcal{L}$.
Then $\mathcal{M}\otimes \mathcal{L}$ is a co-representation of $\D_n\U_n(\mathcal{L})=\mathcal{L}^{\otimes(n-1)}$ with respect to the linear maps (actions)
\[
[- , -]:\mathcal{M}\otimes \mathcal{L} \otimes \mathcal{L}^{\otimes(n-1)} \lra \mathcal{M}\otimes \mathcal{L} \quad \text{and} \quad
[- , -]: \mathcal{L}^{\otimes(n-1)} \otimes \mathcal{M}\otimes \mathcal{L} \lra \mathcal{M}\otimes \mathcal{L}
\]
given by
\begin{align*}
& [m\otimes l, l_1\otimes \dots \otimes l_{n-1}]= m\otimes [l, l_1, \dots, l_{n-1}]  \\
 \text{and} \quad &[l_1\otimes \dots \otimes l_{n-1}, m\otimes l] = - m\otimes [l, l_1, \dots, l_{n-1}].
\end{align*}
 In particular, the co-representation structure on $\mathcal{M}\otimes \mathcal{L}$ over the Leibniz algebra $\mathcal{L}$ is given by
 $[m\otimes l, l_1]=m\otimes [l, l_1]$ and $[l_1, m\otimes l] = - m\otimes [l, l_1]$.

Next, note that
\begin{align*}
{}_{n}CL_k(\U_n(\mathcal{L}), \mathcal{M}) &= CL_k(\D_n\U_n(\mathcal{L}), \mathcal{M}\otimes \mathcal{L}) = CL_k(\mathcal{L}^{\otimes(n-1)}, \mathcal{M}\otimes \mathcal{L}) \\
&= \mathcal{M}\otimes \mathcal{L} \otimes  \big({\mathcal{L}^{\otimes(n-1)}}\big)^{\otimes k} = \mathcal{M}\otimes \mathcal{L} \otimes {\mathcal{L}^{\otimes k(n-1)}},
\end{align*}
whilst
\[
{}_2CL_k(\mathcal{L}, \mathcal{M})=\mathcal{M}\otimes \mathcal{L}\otimes \mathcal{L}^{\otimes k}.
\]
And the map $h_k: \mathcal{M}\otimes \mathcal{L} \otimes {\mathcal{L}^{\otimes k(n-1)}} \lra \mathcal{M}\otimes \mathcal{L}\otimes \mathcal{L}^{\otimes k}$
is given by
\begin{align*}
h_k&(m\otimes l_1^1\otimes l_2^1 \otimes \dots \otimes l_{n-1}^1\otimes \dots\otimes l_1^k\otimes l_2^k \otimes \dots \otimes l_{n-1}^k)\\
&=m\otimes l \otimes [l_1^1, [l_2^1, \dots [l_{n-2}^1, l_{n-1}^1] \dots ]] \otimes \dots \otimes [l_1^k, [l_2^k, \dots [l_{n-2}^k, l_{n-1}^k] \dots ]] \, .
\end{align*}
The details are straightforward but tedious calculations and are left to the reader.
It is clear by the definition of $h_k$ that it is an epimorphism, when $\mathcal{L}$ is a perfect Leibniz algebra.
\end{proof}

\begin{remark}
Given a trivial co-representation $\mathcal{M}$ of a Leibniz algebra $\mathcal{L}$,
there is a homomorphism $\overline{h}_*: {}_{n}HL_*(\U_n(\mathcal{L}), \mathcal{M}) \lra {}_2HL_*(\mathcal{L}, \mathcal{M})$, with $\overline{h}_0$ an epimorphism. Moreover, if $\mathcal{L}$ is a perfect Leibniz algebra, then $\overline{h}_0$ is the identity map and $\overline{h}_1$ is an epimorphism.
\end{remark}

\section{$\U_n^p$ functors and universal central extensions}\label{Section UCE}

\begin{definition} \cite{Ca}
An extension of Leibniz $n$-algebras
\begin{equation}\label{equation CE}
0 \lra M \lra {\ka} \stackrel{\pi} \lra {\ele} \lra 0
\end{equation}
is called central extension of $\ele$, if $M\subseteq \Z(\mathcal{K})$.
 \end{definition}

Let us remark that an extension $0 \to M \to {\ka} \to {\ele} \to 0$ in $\nLb$ is central if and
only if the ideal $[M, {\mathcal{K}}^{n-1}]$ is zero.
Moreover, it is known that the category of Leibniz $n$-algebras is a semi-abelian category and
the definition of central extension agrees with the same notion in the general framework of semi-abelian categories
 \cite{JaKe}, particularly, it is induced by the adjunction between $\nLb$ and its Birkhoff subcategory of vector spaces (see \cite{CKLT}).

\begin{remark}

If (\ref{equation CE}) is a central extension of a Leibniz $n$-algebra $\ele$, then $\pi: {\ka} \to {\ele}$ is a crossed module with the action of $\ele$ on $\ka$ defined in the standard way by choosing pre-images of elements of $\ele$ and tacking the $n$-ary bracket in $\ka$ (see \cite[Example 2 (iii)]{CKL1}).

On the other hand,  given a crossed module of Leibniz $n$-algebras $\pi: {\ka} \to {\ele}$, the kernel of $\pi$ is in the center of $\ka$ thanks to the condition (CM2), and we get a central extension $0 \to \Ker\pi \to {\ka} \stackrel{\pi} \to {\ele} \to 0$  of $\ele$.
\end{remark}

The following example immediately below shows that, unlike the case of the functor $\U_n$, the functor $\D_n$ does not preserve crossed modules.

\begin{example}
 Let $\pi : K \to L$ be the crossed module of Leibniz $n$-algebras provided by the central extension (\ref{equation CE}). Then $\D_n(\pi)=\pi^{\otimes (n-1)}  : \ka^{\otimes (n-1)}  \to \ele^{\otimes (n-1)}$ may not be a crossed module, since the kernel of $\pi^{\otimes (n-1)}$ is not in the center of $\ka^{\otimes (n-1)}$ in general.
\end{example}

\begin{definition} \label{Def UCE}
The central extension $(\ref{equation CE})$ is said to be {\it universal} if for every central extension $ 0 \to M' \to {\ka}' \stackrel{\pi'} \to {\ele} \to 0$ there exists a unique homomorphism $h : {\ka} \to {\ka}'$ such that $\pi'\circ h = \pi$.
\end{definition}

It is well-known  that a Leibniz $n$-algebra ${\ele}$ admits a universal central
extension if and only if ${\ele}$ is perfect. In such a case, the universal central extension can be described via the non-abelian tensor power of $\ele$ (see \cite{Ca, Ca 1}). For the convenience of the reader, we will briefly recall this construction.

Let ${\ele}$ be a Leibniz $n$-algebra. The tensor power ${\ele}^{\otimes n}$ can be endowed with a structure of Leibniz $n$-algebra by means of the following bracket:
\begin{equation} \label{tensor}
\begin{array}{l} [x_{11} \otimes \dots \otimes x_{n1}, x_{12} \otimes \dots \otimes x_{n2}, \dots, x_{1n} \otimes \dots \otimes x_{nn}]  \\
= \displaystyle \sum_{i=1}^n x_{11} \otimes \dots \otimes [x_{i1}, [x_{12}, \dots, x_{n2}], \dots, [x_{1n}, \dots, x_{nn}]] \otimes \dots \otimes x_{n1} \, .
\end{array}
\end{equation}

Now suppose $\delta_2: {\ele}^{\otimes (2n-1)} \to {\ele}^{\otimes n}$ is the second boundary map in the homology complex $_nCL(\ele,\mathbb{K})$, where $\mathbb{K}$ is considered as the trivial co-representation of $\ele$. Hence,
\begin{align*}
\delta_2(x_{1} \otimes \dots \otimes x_{2n-1})=& [x_1, x_2, \dots, x_n]\otimes x_{n+1}\otimes \dots \otimes x_{2n-1}
\\
&-\underset{\scriptsize{1\leq i \leq n}}\sum x_1\otimes \dots \otimes [x_{i}, x_{n+1}, \dots, x_{2n-1}]\otimes \dots\otimes x_n \, .
\end{align*}
Following to \cite{Ca 1}, $\Coker(\delta_2: {\ele}^{\otimes (2n-1)} \to {\ele}^{\otimes n})$
is equipped with a structure of Leibniz $n$-algebra induced by the bracket (\ref{tensor}) defined on ${\ele}^{\otimes n}$.
This Leibniz $n$-algebra is denoted by ${\ele}^{\ast n}$  and it is called the non-abelian tensor power of $\ele$.

Let us denote the image of $x_1 \otimes \dots \otimes x_n \in {\ele}^{\otimes n}$ into ${\ele}^{\ast n}$  by $x_1 \ast \dots \ast x_n$. Then one verifies  (see \cite{Ca}) that the bracket in ${\ele}^{\ast n}$ is given by
\begin{equation}
\begin{array}{l}
[x_{11} \ast \dots \ast x_{n1}, \ x_{12} \ast \dots \ast x_{n2}, \ \dots , \ x_{1n} \ast \dots \ast x_{nn}] \\

= [x_{11}, \dots, x_{n1}] \ast [x_{12}, \dots, x_{n2}] \ast \dots \ast [ x_{1n}, \dots, x_{nn}]
\end{array}
\end{equation}

There is an exact sequence of Leibniz $n$-algebras
\begin{equation*} \label{4 term seq}
0 \lra {_n}HL_1({\ele}) \lra {\ele}^{\ast n} \stackrel{[-,\stackrel{n}{\cdots}, -]}{\lra} {\ele} \to {_n}HL_0({\ele})\lra  0,
\end{equation*}
where ${_n}HL_0({\ele})={_n}HL_0({\ele}, \mathbb{K})= \ele / [\ele^n]$ and ${_n}HL_1({\ele})={_n}HL_1({\ele}, \mathbb{K})$ are the zero and the first homology of the Leibniz $n$-algebra ${\ele}$ with coefficients in the trivial co-representation $\mathbb{K}$, respectively, and both are considered as abelian Leibniz $n$-algebras. In particular, if $\ele$ is a perfect Leibniz $n$-algebra, then ${_n}HL_0({\ele})=0$ and
\begin{equation} \label{central extension}
0 \lra {_n}HL_1({\ele}) \lra {\ele}^{\ast n} \stackrel{[-,\stackrel{n}{\cdots}, -]}{\lra} {\ele} \lra 0
\end{equation}
is the universal central extension of $\ele$.

\

Having in mind  Theorem \ref{Theorem perfect}, by applying the functor $\U_n^p$ to (\ref{central extension}) we get again the exact sequence
\begin{equation} \label{central extension in pLb}
0 \lra \U_n^p({_n}HL_1({\ele}) )\lra \U_n^p({\ele}^{\ast n}) \stackrel{}{\lra} \U_n^p({\ele}) \lra 0
\end{equation}
in the category $\pLb$. The $p$-ary bracket in  $\U_n^p({\ele}^{\ast n} )={\ele}^{\ast n}$ is given by

\begin{align*}
[x_{11} \ast \dots  & \ast x_{1n},x_{21} \ast \dots \ast x_{2n}, \dots, x_{p1} \ast \dots \ast x_{pn}] \\
=[x_{11} \ast & \dots \ast x_{1n}, \dots,  x_{(n-1)1} \ast \dots \ast x_{(n-1)n},[x_{n1} \ast \dots \ast x_{nn}, \dots, x_{(2n-2)1} \ast \dots \ast x_{(2n-2)n},\\
&[ \dots,[ x_{(q-n+1)1} \ast \dots \ast x_{(q-n+1)n}, \dots,  x_{q1} \ast \dots \ast x_{qn}] \dots ]]] \\
=[x_{11} \ast &\dots \ast x_{1n}, \dots,  x_{(n-1)1} \ast \dots \ast x_{(n-1)n},[x_{n1} \ast \dots \ast x_{nn}, \dots, x_{(2n-2)1} \ast \dots \ast x_{(2n-2)n},\\
&[ \dots,[[ x_{(q-n+1)1} , \dots , x_{(q-n+1)n}]
\ast  \dots \ast  [x_{q1}, \dots, x_{qn}]] \dots]]] \\
= \dots  = &
[x_{11}, \dots, x_{1n}]\ast \dots \ast [x_{(n-1)1}, \dots, x_{(n-1)n}] \ast \left[  \left[x_{n1}, \dots , x_{nn}\right], \dots, \right.\\
&\left.  \left. \left[x_{(2n-2)1}, \dots, x_{(2n-2)n} \right], \left[ \dots \left[ x_{(p-n+1)1}, \dots, x_{(p-n+1)n} \right], \dots, \left[ x_{p1} , \dots , x_{pn} \right] \right] \dots  \right] \right].
\end{align*}

Consequently, the map ${\ele}^{\ast n}  \stackrel{[-, \stackrel{n}{\cdots},-]}\lra {\ele}$, $x_1 \ast  \dots \ast x_n \mapsto [x_1, \dots, x_n]$ is a
homomorphism of Leibniz $p$-algebras and (\ref{central extension in pLb}) is  a central extension of Leibniz $p$-algebras.

Thanks again to Theorem \ref{Theorem perfect}, ${\ele}= \U_n^p({\ele})$ is perfect as Leibniz $p$-algebra and admits the universal central extension
\[
0 \lra {_{p}HL_1({\ele})} \lra {\ele}^{\ast p}  \stackrel{[-, \stackrel{p}{\cdots},-]}\lra {\ele} \lra 0
\]
in the category ${\pLb}$.
 Since
 \[
 0 \lra {_{n}HL_1({\ele})} \lra {\ele}^{\ast n}  \stackrel{[-, \stackrel{n}{\cdots},-]}\lra {\ele} \lra 0,
 \]
  endowed with the above structure of Leibniz $p$-algebra, is a central extension in ${\pLb}$, then there exists a unique homomorphism of Leibniz $p$-algebras
\[
\phi : {\ele}^{\ast p} \lra {\ele}^{\ast n}
\]
which is given by $\phi(x_1 \ast \dots \ast x_p) = x_1\ast \dots \ast x_{n-1} \ast
 \left[x_{n}, \dots, x_{2n-2},\left[ \dots, \left[x_{p-n+1},
\dots,x_{p}\right] \right] \right]$,
and makes commutative the following diagram in $\pLb$:
\begin{equation} \label{diagram}
\xymatrix{
 0  \ar[r] & {_{p}}HL_1({\ele})  \ar[r] \ar[d] & {\ele}^{\ast p} \ar[r]^{\quad\quad  [-,\stackrel{p}\dots,-] \quad\quad} \ar[d]_{\phi} & {\ele} \ar[r] \ar@{=}[d] &0\\
0  \ar[r] & {_{n}}HL_1({\ele})  \ar[r]  & {\ele}^{\ast n} \ar[r]^{\quad\quad  [-,\stackrel{n}\dots,-] \quad\quad}  & {\ele} \ar[r] &0
}
\end{equation}
Moreover, obviously $\phi$ is a surjective homomorphism.

\begin{corollary} If ${\ele}$ is a perfect Leibniz $p$-algebra, then
\begin{align*}
{\ele} \ast \stackrel{n-1}\dots  \ast {\ele} \ast \Z({\ele}) \ast {\ele} \ast \stackrel{p-n}\dots \ast
{\ele}\  + \ {\ele} \ast& \stackrel{n}\dots \ast {\ele} \ast \Z({\ele}) \ast {\ele} \ast \stackrel{p-n-1}\dots \ast
{\ele}   \\
& + \  \dots \ + {\ele} \ast \stackrel{p-1}\dots \ast {\ele} \ast  Z({\ele}) \ \subseteq \ {\Ker}(\phi) \ .
\end{align*}
\end{corollary}


\

\subsection*{Acknowledgements}

Authors were supported by  Agencia Estatal de Investigaci\'on (Spain), grant PID2020-115155GB-I00.
Emzar Khmaladze was supported by Shota Rustaveli National Science Foundation of Georgia, grant FR-22-199. He is grateful to the University of Santiago de Compostela for hospitality.


\begin{thebibliography}{99}

\bibitem{BSO} {M. R. Bremner and J. S\'anchez-Ortega}, {\it  Leibniz triple systems}, Commun. Contemp. Math.
  {\bf 16} (1) (2014),  1350051, 19 pp.
  
  \bibitem{CCGLO} {L. M. Camacho, J. M. Casas, J. R. G\'omez, M. Ladra and B. A. Omirov}, {\it On nilpotent Leibniz $n$-algebras},
  J. Algebra Appl. {\bf 11}  (2012), no. 2, 1250060 (17 pp.).
  
  \bibitem{Ca 1} {J. M. Casas}, {\it Homology with trivial coefficients of Leibniz $n$-algebras}, Comm.  Algebra {\bf 31}   (2003), no. 3, 1377--1386.

\bibitem{Ca} {J. M. Casas}, {\it A non-abelian tensor product and universal central extension of Leibniz $n$-algebras},
 Bull. Belg. Math. Soc.-Simon Stevin {\bf 11}  (2004), no. 2, 259--270.



\bibitem{Ca 2} {J. M. Casas}, {\it Homology with  coefficients of Leibniz $n$-algebras}, C. R. Acad. Sci. Paris, Ser. I  {\bf 347} (2009), 595--598.



\bibitem{CKL1} {J. M. Casas, E. Khmaladze and M. Ladra}, {\it  Crossed modules for Leibniz $n$-algebras},
Forum Math. {\bf 20}  (2008), 841--858.

\bibitem{CKLT} {J. M. Casas, E. Khmaladze, M. Ladra and T. Van der Linden}, {\it  Homology and central extensions of Leibniz and Lie $n$-algebras},
Homology Homotopy Appl. {\bf 13(1)} (2011), 59--74.



\bibitem{CLP} {J. M. Casas, J.-L. Loday and T. Pirashvili}, {\it  Leibniz $n$-algebras}, Forum Math. {\bf 14} (2002), no. 2,  189--207.

\bibitem{DT} {Y. Daletskii and L. A. Takhtajan}, {\it  Leibniz and Lie algebra structures for Nambu algebra}, Letters in Math. Physics {\bf 39} (1997),  127--141.

 \bibitem{Fi} {V. T. Filippov}, {\it n-Lie algebras}, Sib. Mat. Zh. {\bf 26} (1985), no. 6, 126--140.

  \bibitem{Ga} {P. Gautheron}, {\it Some remarks concerning Nambu mechanics}, Letters in Math. Physics {\bf 37} (1996),  103--116.

 \bibitem{JaKe}    G. Janelidze and G.M. Kelly, {\it Galois theory and a general notion of central
extension}, J. Pure Appl. Algebra {\bf 97(2)} (1994),  135--161.

 \bibitem{Rustam} M. S. Kim, R. Turdibaev, {\it  Some theorems on Leibniz $n$-algebras from the category ${\sf U}_n({\sf Lb})$}, Algebra Colloq. {\bf 28} (2021), no. 1, 155--168.


 \bibitem{Li} { W. Lister}, {\it A structure theory of Lie triple systems}, Trans. Amer. Math. Soc. {\bf 72} (1952), 217--242.

 \bibitem{Lo} {J.-L. Loday}, {\it Cyclic homology}, Grundl. Math. Wiss. Bd. {\bf 301}, Springer 1992.

 \bibitem{LP} {J.-L. Loday and T. Pirashvili}, {\it Universal enveloping algebras of Leibniz
 algebras and (co)homology},  Math. Ann. {\bf 296} (1993), no. 1,  139--158.

 \bibitem{LR} {J.-L. Loday, M. O. Ronco}, {\it Trialgebras and families of polytopes}, Homotopy theory: relations
 with algebraic geometry, group cohomology, and algebraic $K$-theory, 369--398, Contemp. Math., {\bf 346},
 Amer. Math. Soc., Providence, RI, 2004.

\bibitem{Na} {Y. Nambu}, {\it Generalized Hamiltonian dynamics}, Phys. Rev. D   {\bf 7}  (1973), no. 3, 2405--2412.

\bibitem{Om} {B. A. Omirov}, {\it Conjugacy of Cartan subalgebras of complex finite-dimensional Leibniz algebras}, J. Algebra  {\bf 302}  (2006),  887--896.
\end{thebibliography}
\end{document}